\documentclass[12pt]{article}
\usepackage{mathrsfs}
\usepackage{bbm}

\usepackage{amssymb,a4}
\usepackage{amsmath,amsfonts,amssymb,amsthm}

\newcommand{\1}{\mathbbm {1}}
\newcommand{\Z}{{\mathbb Z}}
\newcommand{\C}{{\mathbb C}}
\newcommand{\R}{{\mathbb R}}
\newcommand{\h}{{\mathfrak h}}
\newcommand{\wh}{{\widehat{\mathfrak h}}}

\def\<{\langle}
\def\>{\rangle}
\def\a{\alpha}
\def\b{\beta}
\def\g{\mathfrak g}
\def\h{\mathfrak h}
\def\wt{\rm wt}
\newcommand{\mraff}{\mathrm{aff}}

\newcommand{\e}{\epsilon}

\newtheorem{thm}{Theorem}[section]

\newtheorem{lem}[thm]{Lemma}
\newtheorem{cor}[thm]{Corollary}

\begin{document}

\begin{center}
{\Large {\bf  On $C_2$-cofiniteness of parafermion vertex operator algebras
}} \\
\vspace{0.5cm}


Chongying Dong\footnote{Supported by NSF grants,
and  a Faculty research grant from  the
University of California at Santa Cruz.}
\\
Department of Mathematics, University of
California, Santa Cruz, CA 95064\\
\vspace{.1cm}
\& School of Mathematics,  Sichuan University, Chengdu, 610065 China\\
\vspace{.2 cm}
 Qing Wang\footnote{Supported by China NSF grants(No.10931006, No.10926040),
 and Natural Science Foundation of Fujian Province, China(No.2009J05012).}\\
School of Mathematical Sciences, Xiamen University,  Xiamen,
361005 China
\end{center}




\begin{abstract}
It is proved that the regularity of parafermion vertex operator algebras
associated to  integrable highest weight modules for affine Kac-Moody algebra
$A_1^{(1)}$ implies the $C_2$-cofiniteness of parafermion vertex operator algebras associated to integrable highest weight modules for any
affine Kac-Moody algebra. In particular,
the parafermion vertex operator algebra associated to an
integrable highest weight module of small level for any affine Kac-Moody
algebra is $C_2$-cofinite and has only finitely many
irreducible modules. Also, the parafermion vertex operator
algebras with level 1 are determined explicitly.
\end{abstract}


\section{Introduction}
\def\theequation{1.\arabic{equation}}
\setcounter{equation}{0}

This paper is devoted to the study of the $C_2$-cofiniteness of parafermion
 vertex operator algebra $K(\g,k)$ associated to integrable highest weight
module of level $k$  for affine Kac-Moody algebra $\widehat\g$,
where $\g$ is a finite dimensional simple Lie algebra. It is
established that the regularity of $K(sl_2,k)$ for all $k$ implies
the $C_2$-cofiniteness of $K(\g,k)$ for all $k.$ In particular,
the $C_2$-cofiniteness of $K(\g,k)$ is obtained if $\g$ is of
$ADE$ type and $k\leq 6,$ $\g$ is of type $G_2$ and $k\leq 2,$ and
$\g$ is of type $B_l,C_l,F_4$ and $k\leq 3.$ The structure and
representation theory of $K(\g,1)$ is analyzed in details. It
turns out that $K(\g,1)$ for $\g$ being of type $B_l, C_l,
F_4,G_2$ is isomorphic to the parafermion vertex operator algebra
of type $A$ with level $k=2$ or $3.$

The origin of the parafermion vertex operator algebras is the
$Z$-algebras developed in \cite{LP}, \cite{LW1}, \cite{LW2} in the
construction of integrable highest weight modules for affine
Kac-Moody algebras. It was investigated in \cite{ZF} how the
$Z$-algebras and $Z$-operators lead to a new conformal field
theory -- parafermion conformal field theory. This further
stimulated the development of the theory of generalized vertex
operator algebras \cite{DL} which, in turns, provides a framework
for the parafermion conformal field theory. It is well-known that
the affine Kac-Moody Lie algebra $\widehat\g$ contains a
Heisenberg algebra. As a result, the vertex operator algebra
${\cal L}(k,0)$ associated to the highest weight module of level
$k$ for $\widehat\g$ contains the Heisenberg vertex operator
algebra as a subalgebra. The parafermion vertex operator algebra
$K(\g,k)$ is the commutant  \cite {FZ}, \cite{LL} of the
Heisenberg vertex operator algebra in ${\cal L}(k,0),$ and is a
special kind of coset construction \cite{GKO}.

Some aspects of both the structure and representation theory of
parafermion vertex operator algebras have been studied in
\cite{DLY1}-\cite{DLWY},\cite{DW}. In particular, a set of
generators for general parafermion vertex operator algebra
$K(\g,k)$ is obtained. It is also shown that $K(\g,k)$ is
generated by $K(sl_2,k_{\alpha})$ for positive roots $\alpha,$
where $k_\alpha \in \{k,2k,3k\}$ is determined by the squared
length of $\alpha$ \cite{DW}. This suggests that the structure and
representation theory for general parafermion vertex operator
algebra $K(\g,k)$ can be understood by using the structure and
representation theory of $K(sl_2,k_{\alpha})$(see \cite{DW} for
details).

The rationality \cite{Z},\cite{DLM2} on the semisimplicity of the
admissible module category and $C_2$-cofiniteness \cite{Z} on the
cofiniteness of certain subspace of the vertex operator algebra
are perhaps two most important concepts in the representation
theory of vertex operator algebras. It was proved in \cite{L2} and
\cite{ABD} that the rationality together with $C_2$-cofiniteness
is equivalent to the regularity \cite{DLM1} which says that any
weak module is a direct sum of irreducible ordinary modules.  Many
well known vertex operator algebras such as vertex operator
algebras associated to positive definite even lattices, integrable
highest weight modules for the affine Kac-Moody algebras, highest
weight modules associated to the minimal series for the Virasoro
algebras are regular \cite{DLM1}. It is natural to expect the
rationality and $C_2$-cofiniteness of the parafermion vertex
operator algebra $K(\g,k)$ according to the general principle in
the coset theory:
 the commutant of a regular vertex operator subalgebra in a regular
vertex operator algebra is again regular. On the surface,
$K(\g,k)$ is the commutant of the Heisenberg vertex operator
algebra which is neither rational nor $C_2$-cofinite in regular
vertex operator algebra ${\cal L}(k,0).$ But  the parafermion
vertex operator algebra $K(\g,k)$ can also be realized as the
commutant of lattice vertex operator algebra $V_{\sqrt{k}Q_L}$ in
${\cal L}(k,0)$, where $Q_L$ is the positive definite even lattice
generated by the long roots of $\g$ (for example, see
\cite{K},\cite{DL},\cite{DLY2}). There is no doubt that the
representation theory of $K(\g,k)$ would be rich and interesting.

As we have already mentioned that we hope some important
properties such as $C_2$-cofiniteness and rationality for general
parafermion vertex operator algebra $K(\g,k)$ can be obtained by
studying the simplest parafermion vertex operator algebra
associated to  $\widehat{sl_2}$. In this paper, we succeed in
proving that the $C_2$-cofiniteness for general parafermion vertex
operator algebra $K(\g,k)$ is completely determined by the
regularity of $K(sl_2,k)$. It was shown in \cite{DLY2} that
$K(sl_2,k)$ is rational and $C_2$-cofinite for $k\leq 6$, thus the
$C_2$-cofiniteness for general parafermion vertex operator algebra
$K(\g,k)$ of small level $k$ follows. The main idea in proving
this result is to use the generating result given in \cite{DW},
the property of the hermitian operator in the unitary
representation theory for the Virasoro algebra and the approach in
\cite{L2} to the maximal weak module in a completion of an
ordinary module.

Although the generators of general parafermion vertex operator
algebra $K(\g,k)$ are determined in \cite{DW}, it is far from over
to understand the structure theory for $K(\g,k)$ completely. The
main challenging is to figure out how the generators act on each
other or how the generators are related. As an experiment, we
study the simplest case $K(\g,k)$ with $k=1$ in great details. It
is clear that $K(\g,1)=\C$ if $\g$ is of $ADE$ type. So the main
discussion is put on the non-simply laced simple Lie algebras,
i.e., those of type $B_l,\; C_l,\; F_4,$ and $G_2$. A critical
observation in the analysis of $K(\g,1)$ for $\g$ being non-simply
laced simple Lie algebra is Lemma \ref{l4.2} which tells us
the relation between the generators associated to the short roots.
The result can be summarized as follows: $K(B_l,1)$ is isomorphic
to $K(A_1,2)$ which is the Virasoro vertex operator algebra
associated to the irreducible highest weight module for the
Virasoro algebra with central charge $\frac{1}{2},$ $K(C_l,1)$ is
isomorphic to $K(A_{l-1},2),$ $K(F_4,1)$ is isomorphic to
$K(A_2,2)$ and $K(G_2,1)$ is isomorphic to $K(A_1,3).$

The paper is organized as follows. In Section 2, we fix the
setting, recall the definition  of parafermion vertex operator
algebra $K(\g,k)$, and present some results on parafermion vertex
operator algebra $K(\g,k)$ from \cite{DLY1}, \cite{DLY2},
\cite{DLWY}, \cite{DW}. In Section 3, we prove the main result in
this paper. That is, the regularity of $K(sl_2,k)$ implies the
$C_2$-cofiniteness of $K(\g,k^{'})$, where $k^{'}$ depends on $k$
and the Lie algebra $\g.$ The proof involves with the positive
definite hermitian form and the nondegenerate symmetric invariant
 bilinear form on $K(\g,k^{'}).$ While the hermitian form allows us to study
certain hermitian operators on a completion of $K(\g,k^{'}),$ the
bilinear form gives us the identification of $K(\g,k^{'})$ with
its contragredient module. The semisimplity of the hermitian
operator associated to any root of $\g$ forces the maximal weak
$K(\g,k^{'})$-module in the completion of $K(\g,k^{'})$ to be
$K(\g,k^{'})$ itself. The $C_2$-cofiniteness of $K(\g,k^{'})$ is
immediate by a result in \cite{L2}.  In Section 4, we discuss the
structure of $K(\g,1)$ as a starting point to the follow-up work
on the representation theory of parafermion vertex operator
algebras.

We expect the reader to be familiar with the elementary theory of
vertex operator algebras as found, for example, in \cite{FLM} and
\cite{LL}.

\section{Parafermion vertex operator algebras $K(\g,k)$}
\label{Sect:V(k,0)}
\def\theequation{2.\arabic{equation}}
\setcounter{equation}{0}

We are working in the setting of \cite{DW}. Fix a finite
dimensional simple Lie algebra $\g$ with a Cartan subalgebra $\h.$
We use $\Delta$ and $Q$ to denote the corresponding root system
and root lattice, respectively. We also fix an invariant symmetric
nondegenerate bilinear form $\<,\>$ on $\g$ so that $\<\a,\a\>=2$
if $\alpha$ is a long root, where we have identified $\h$ with
$\h^*$ via $\<,\>.$ As in \cite{H}, for $\alpha\in \h^*$, we
denote its image in $\h$ by  $t_\alpha,$ that is,
$\alpha(h)=\<t_\alpha,h\>$ for any $h\in\h.$  Let
$\{\alpha_1,\cdots,\alpha_l\}$ be the simple roots and denote the
highest root by $\theta.$

Let $\g_{\alpha}$ denote the root space associated to the root
$\a\in \Delta.$ For $\alpha\in \Delta_+$, we fix $x_{\pm
\alpha}\in \g_{\pm \alpha}$ and
$h_{\alpha}=\frac{2}{\<\a,\a\>}t_\alpha\in \h$ such that
 $[x_\a,x_{-\a}]=h_{\a}, [h_\a,x_{\pm \a}]=\pm 2x_{\pm\a}.$ That
is, $\g^{\a}=\C x_{\a}+\C h_{\alpha}+\C x_{-\alpha}$ is isomorphic
to $sl_2$ by sending $x_\a$ to $\left(\begin{array}{ll} 0 & 1\\ 0 &
0\end{array}\right),$ $x_{-\a}$ to $\left(\begin{array}{ll} 0 & 0\\
1 & 0\end{array}\right)$ and $h_\a$ to $\left(\begin{array}{ll} 1 &
0\\ 0 & -1\end{array}\right).$ Then
$\<h_\a,h_\a\>=\frac{4}{\<\alpha,\alpha\>}$ and
$\<x_{\a},x_{-\a}\>=\frac{2}{\<\alpha,\alpha\>}$
for all
$\alpha\in \Delta.$

Let $\widehat{\mathfrak g}= \g \otimes \C[t,t^{-1}] \oplus \C K$
be the corresponding affine Lie algebra. Fix a positive integer
$k$ and let
\begin{equation*}
V(k,0) = V_{\widehat{\g}}(k,0) = Ind_{\g \otimes
\C[t]\oplus \C K}^{\widehat{\g}}\C
\end{equation*}
be the induced module where ${\g} \otimes \C[t]$ acts as $0$ and $K$ acts as $k$ on $\1=1$.
Then $V(k,0)$ is a
vertex operator algebra generated by $a(-1)\1$ for $a\in \g$ such
that
$$Y(a(-1)\1,z) = a(z)=\sum_{n \in \Z} a(n)z^{-n-1}$$
where $a(n)=a\otimes t^n$, with the
vacuum vector $\1$ and the Virasoro vector
\begin{align*}
\omega_{\mraff} &= \frac{1}{2(k+h^{\vee})} \Big(
\sum_{i=1}^{l}h_i(-1)h_i(-1)\1 +\sum_{ \alpha\in\Delta}
\frac{\<\a,\a\>}{2}x_{\alpha}(-1)x_{-\alpha}(-1)\1
\Big)
\end{align*}
of central charge $\frac{k\dim \g}{k+h^{\vee}}$ (e.g. \cite{FZ},
\cite{K}, \cite[Section 6.2]{LL}), where $h^{\vee}$ is the dual
Coxeter number of $\g$ and $\{h_i|i=1,\cdots,l\}$ is an
orthonormal basis of $\mathfrak h.$

Let $M(k)$ be the vertex operator subalgebra of $V(k,0)$
generated by $h(-1)\1$ for $h\in \mathfrak h$ with the Virasoro
element
$$\omega_{\mathfrak h} = \frac{1}{2k}
\sum_{i=1}^{l}h_i(-1)h_i(-1)\1$$
of central charge $l$, where $\{h_1,\cdots h_l\}$ is an
orthonormal basis of $\mathfrak h$ as before.

As usual we denote the component operators of $Y(u,z)$ for $u\in
V$ by $u_n$ for any vertex operator algebra $V.$ That is,
$Y(u,z)=\sum_{n\in \Z}u_nz^{-n-1}.$ In the case $V=V(k,0)$ and
$u=a(-1)\1$ for $a\in \g$, we see that $(a(-1)\1)_n=a(n).$ So in
the rest of paper, we will use both $a(n)$ and $(a(-1)\1)_n$ for
$a\in \g$ and use $u_n$ only for general $u$ without further
explanation.

We denote the unique irreducible quotient $\widehat{\g}$-module of
$V(k,0)$ by ${\mathcal{L}}(k,0).$ Then ${\mathcal{L}}(k,0)$ is a
simple, rational  vertex operator algebra. Moreover, the image of $M(k)$
in ${\mathcal{L}}(k,0)$ is isomorphic to $M(k)$ and will be denoted by $M(k)$
again. Set
\begin{equation*}
 K(\g,k)=\{v \in {\mathcal{L}}(k,0)\,|\, h(m)v =0
\text{ for }\; h\in {\mathfrak h},
 m \ge 0\}.
\end{equation*}
Then $K(\g,k)$ which is the space of highest weight vectors
with highest weight $0$ for $\wh$
is the commutant of $M(k)$ in ${\mathcal{L}}(k,0)$
and is called the parafermion vertex operator algebra associated
to the irreducible highest weight module ${\mathcal{L}}(k,0)$ for
$\widehat{\g}.$ The Virasoro element of $K(\g,k)$ is given by
$$\omega =\omega_{\mraff} - \omega_{\mathfrak h},$$
where we still use $\omega_{\mraff}, \omega_{\mathfrak h}$ to
denote their images in ${\mathcal{L}}(k,0)$.

Let $\alpha\in \Delta$ and set
$k_{\alpha}=\frac{2}{\<\alpha,\alpha\>}k.$ Note that
$\widehat{\mathfrak g^{\alpha}}$ is a subalgebra of
$\widehat{\mathfrak g}$ and ${\mathcal{L}}(k,0)$ is an integrable
$\widehat{\mathfrak g^{\alpha}}$-module of level $k_{\a}.$

For $\alpha\in \Delta$, we set
\[
\begin{split}
\omega_{\alpha} &=\frac{1}{2k_{\alpha}(k_{\alpha}+2)}( -k_{\a} h_\alpha(-2)\1
-h_\alpha(-1)^{2}\1 +2k_{\alpha}x_{\alpha}(-1)x_{-\alpha}(-1)\1)\\
&=\frac{1}{2k_{\alpha}(k_{\alpha}+2)}(-h_\alpha(-1)^{2}\1 +
k_{\alpha}x_{\alpha}(-1)x_{-\alpha}(-1)\1+k_{\alpha}x_{-\alpha}(-1)x_{\alpha}(-1)\1)
\end{split}
\]
\[
\begin{split}
W_{\alpha}^3 &= k^2_{\a} h_\alpha(-3)\1 + 3 k_{\a} h_\alpha(-2)h_\alpha(-1)\1
+
2h_\alpha(-1)^3\1 -6k_{\a} h_\alpha(-1)x_{\alpha}(-1)x_{-\alpha}(-1)\1 \\
& \quad +3 k^2_{\a}x_{\alpha}(-2)x_{-\alpha}(-1)\1 -3
k^2_{\a}x_{\alpha}(-1)x_{-\alpha}(-2)\1.
\end{split}
\]
It is easy to see that $\omega_{\a}=\omega_{-\a}$ and
$W_{-\alpha}^3=-W_{\alpha}^3$ for $\alpha\in\Delta.$ A straightforward
verification shows that
$$\omega=
\sum_{\alpha\in \Delta_+}\frac{k(k_{\alpha}+2)}{k_{\alpha}(k+h^{\vee})}\omega_{\alpha}.$$

The following result can be found in \cite{ABD}, \cite{DLY2},
\cite{DLWY} and \cite{DW}.
\begin{thm}\label{t1} (1) Vertex operator algebra $K(\g,k)$ is generated by
$\omega_{\a}, W_{\alpha}^3$ for $\alpha\in \Delta_{+}.$

(2) $K(sl_2,k)$ is a simple rational and $C_2$-cofinite for $k\leq 6.$ That is,
$K(sl_2,k)$ is regular for such $k.$

(3) Let $P_{\alpha}$ be the vertex operator subalgebra of
$K(\g,k)$ generated by $\omega_{\a}, W^3_{\alpha}.$ Then
$P_{\alpha}$ is isomorphic to $K(sl_2,k_{\alpha})$ as vertex
operator algebras.
\end{thm}

\section{$C_2$-cofiniteness of $K(\g,k)$}
\def\theequation{3.\arabic{equation}}
\setcounter{equation}{0}

We need to recall some definitions from \cite{Z}, \cite{DLM1},
\cite{DLM2}. A vertex operator algebra $V$ is called
$C_2$-cofinite if $\dim V/C_2(V)<\infty$, where $C_2(V)$ is a
subspace of $V$ spanned by $u_{-2}v$ for $u,v\in V.$ $V$ is called
regular if any weak module is a direct sum of irreducible ordinary
modules. $V$ is called rational if any admissible module is
completely reducible. It was proved in \cite{L2} and \cite{ABD}
that the rationality together with $C_2$-cofiniteness is
equivalent to the regularity. The most important property of
$C_2$-cofinite vertex operator algebra is that such vertex
operator algebra is finitely generated and has a PBW type spanning
set \cite{GN}. Furthermore, $C_2$-cofinite vertex operator algebra
has only finitely many irreducible admissible modules up to
isomorphism and each irreducible admissible module is ordinary
\cite{L2}, \cite{GN}.

In this section, we prove the main theorem of this paper.
\begin{thm}\label{t2} Let $q$ be a positive integer. If $K(sl_2,k)$ is rational and $C_2$-cofinite for $k\leq q,$ then
$K(\g,k)$ is $C_2$-cofinite if $\g$ is ADE type and $k\leq q,$
$\g$ is type $G_2$ and $k\leq [q/3],$ and $\g$ is other type and
$k\leq [q/2]$, where $[r]$ denotes the maximal integer less than
or equal to $r$ for any real number $r$.
\end{thm}
So the $C_2$-cofiniteness for general parafermion vertex operator algebra
$K(\g,k)$ follows from the regularity of vertex operator algebra
$K(sl_2,k)$ for all $k.$ It is expected that the rationality of general parafermion
vertex operator algebra also follows the regularity of $K(sl_2,k)$ for all
$k.$ So the study of simplest parafermion vertex operator algebras
$K(sl_2,k)$ is crucial in understanding the general parafermion vertex operator algebras.

Combining Theorem \ref{t1}, Theorem \ref{t2} and \cite[Proposition
5.7]{ABD}, we immediately have
\begin{cor} $K(\g,k)$ is $C_2$-cofinite
 if $\g$ is ADE type and $k\leq 6,$
$\g$ is type $G_2$ and $k\leq 2,$ and $\g$ is other type and
$k\leq 3.$ In particular, there are only finitely many irreducible
modules up to isomorphism and each irreducible weak module is
ordinary for such vertex operator algebra.
\end{cor}

In order to prove the theorem, we need some facts on the invariant
bilinear form on $K(\g,k)$ and some related results from
\cite{L1}-\cite{L2}. Note that $K(\g,k)=\oplus_{n\geq 0}K(\g,k)_n$
with $K(\g,k)_0=\C \1$ and $K(\g,k)_1=0.$ It follows from
\cite{L1} that there is a unique nondegenerate symmetric bilinear
form \cite{FHL} $\<,\>$ on $K(\g,k)$ such that
$$\<\1,\1\>=1$$
$$\<Y(u,z)v,w\>=\<v,Y(e^{zL(1)}(-z^{-2})^{L(0)}u,z^{-1})w\>$$
for $u,v,w\in K(\g,k)$, where $L(n)$ are the component operator of
$Y(\omega,z)=\sum_{n\in\Z}L(n)z^{-n-2}.$ As a result, $K(\g,k)$ is
isomorphic to its contragredient module
$K(\g,k)'=\sum_{n\in\Z}K(\g,k)_n^*$, where $K(\g,k)_n^*$ is the
dual space of $K(\g,k)_n$ \cite{FHL}. Denote by
$\overline{K(\g,k)}$ the direct product of $K(\g,k)_n$ for
$n\in\Z.$ Each vector in $\overline{K(\g,k)}$ can be written as
$$(v^0,v^1,...,v^n,...),$$
where $v^i\in K(\g,k)_i.$ Then $\overline {K(\g,k)}$ can be
identified with $K(\g,k)^*$ naturally by using the bilinear form
$\<,\>.$

Recall from \cite{B} the Lie algebra
$\widehat{K(\g,k)}=K(\g,k)\otimes \C[t^{\pm 1}]/D(K(\g,k)\otimes
\C[t^{\pm 1}])$, where $D=L(-1)\otimes 1+1\otimes t\frac{d}{dt}.$
Let $a_{(n)}$ be the image of $a\otimes t^n$ in
$\widehat{K(\g,k)}$ for $a\in K(\g,k)$ and $n\in \Z.$ The Lie
bracket in $\widehat{K(\g,k)}$ is defined as follows:
$$[a_{(m)},b_{(n)}]=\sum_{i=0}^{{\wt} a-1}(a_ib)_{(m+n-i)}$$
for homogeneous $a,b\in K(\g,k)$ and $m,n\in\Z.$ Then $K(\g,k)$ is
a $\widehat{K(\g,k)}$-module such that $a_{(n)}$ acts as $a_n$ and
extend the action to  $\overline {K(\g,k)}$ in an obvious way to
make $\overline {K(\g,k)}$ a $\widehat{K(\g,k)}$-module
\cite{FHL}, \cite{L2}. Let ${\mathcal{D}}(K(\g,k))$ be the
subspace of  $\overline{K(\g,k)}$ consisting of vectors $u$ such
that $a_{(n)}u=0$ for $a\in  K(\g,k)$ and $n$ sufficiently large.
Then ${\mathcal{D}}(K(\g,k))$ is a weak $K(\g,k)$-module and
$K(\g,k)$ is a submodule of ${\mathcal{D}}(K(\g,k))$ \cite{L2}.

We also need to use a positive definite hermitian form on
$K(\g,k)$ in the proof of theorem. It is well known that
${\mathcal{L}}(k,0)$ is a unitary representation for the affine
Lie algebra $\widehat{\g}$ \cite{K}. In fact, there is a unique
positive definite hermitian form $(,)$ on ${\mathcal{L}}(k,0)$
such that
$$(\1,\1)=1,$$
$$(h_{\alpha}(m)u,v)=(u,h_{\alpha}(-m)v),$$
$$(x_{\alpha}(m)u,v)=(u,x_{-\alpha}(-m)v)$$
for $\alpha\in \Delta$ and $u,v\in {\mathcal{L}}(k,0), m\in\Z$.
Clearly, the restriction of  $(,)$ to $K(\g,k)$ defines a positive
definite hermitian form on  $K(\g,k).$

Set
$$Y(\omega_{\alpha},z)=\sum_{n\in\Z}L_{\alpha}(n)z^{-n-2}$$
for $\alpha\in \Delta_+$ and let $Vir_{\alpha}$ be the Virasoro algebra
generated by the component operators of $Y(\omega_{\alpha},z).$ Then one can
compute that
\[
\begin{split}
L_{\alpha}(n)&=\frac{1}{2k_{\alpha}(k_{\alpha}+2)}
\sum_{s+t=n}\left(-\mbox{$\circ\atop\circ$}h_\alpha(s)h_{\alpha}(t)\mbox{$\circ\atop\circ$}
+k_{\alpha}\mbox{$\circ\atop\circ$}x_{\alpha}(s)x_{-\alpha}(t)\mbox{$\circ\atop\circ$}
+k_{\alpha}\mbox{$\circ\atop\circ$}x_{-\alpha}(s)x_{\alpha}(t)\mbox{$\circ\atop\circ$}\right)
\end{split}
\]
for $n\in \Z$, where
\[
\mbox{$\circ\atop\circ$}u(s)v(t)\mbox{$\circ\atop\circ$}=
\left\{\begin{array}{ll}
u(s)v(t) & {\rm if}\ \ s<0\\
v(t)u(s)&{\rm if}\ \ s\ge 0
\end{array}\right.
\]
for $u,v\in \g$ and $s,t\in\Z.$

\begin{lem}\label{l1} ${\mathcal{L}}(k,0)$ is a unitary representation for
the Virasoro algebra $Vir_{\alpha}$ for all $\alpha\in \Delta_+.$
That is,
$$(L_{\alpha}(m)u,v)=(u,L_{\alpha}(-m)v)$$
for all $m\in\Z.$ In particular, $L_{\alpha}(0)$ is a hermitian
operator on ${\mathcal{L}}(k,0)_n$ for all $n\in \Z,$ and
$K(\g,k)$ is a unitary representation of $Vir_{\alpha}$ for all
$\alpha\in \Delta_+.$
\end{lem}

\begin{proof}
 Let $u,v\in {\mathcal{L}}(k,0)$ and $n\in\Z.$ Then
\[
\begin{split}
 & 2k_{\alpha}(k_{\alpha}+2)(L_{\alpha}(n)u,v)\\
&\ \  =\sum_{s+t=n,s<0}(-h_\alpha(s)h_{\alpha}(t)u+k_{\alpha}x_{\alpha}(s)x_{-\alpha}(t)u+k_{\alpha}x_{-\alpha}(s)x_{\alpha}(t)u,v)\\
& \ \ \ \ +\sum_{s+t=n,s\geq 0}(-h_\alpha(t)h_{\alpha}(s)u+k_{\alpha}x_{-\alpha}(t)x_{\alpha}(s)u+k_{\alpha}x_{\alpha}(t)x_{-\alpha}(s)u,v)\\
&\ \  =\sum_{s+t=n,s<0}(u, -h_\alpha(-t)h_{\alpha}(-s)v+k_{\alpha}x_{-\alpha}(-t)x_{\alpha}(-s)v+k_{\alpha}x_{\alpha}(-t)x_{-\alpha}(-s)v)\\
& \ \ \ \ +\sum_{s+t=n,s\geq 0}(u,-h_\alpha(-s)h_{\alpha}(-t)v+k_{\alpha}x_{\alpha}(-s)x_{-\alpha}(-t)v+k_{\alpha}x_{-\alpha}(-s)x_{\alpha}(-t)v)\\
&\ \  =\sum_{s+t=n,s\leq 0}(u, -h_\alpha(-t)h_{\alpha}(-s)v+k_{\alpha}x_{-\alpha}(-t)x_{\alpha}(-s)v+k_{\alpha}x_{\alpha}(-t)x_{-\alpha}(-s)v)\\
& \ \ \ \ +\sum_{s+t=n,s>0}(u,-h_\alpha(-s)h_{\alpha}(-t)v+k_{\alpha}x_{\alpha}(-s)x_{-\alpha}(-t)v+k_{\alpha}x_{-\alpha}(-s)x_{\alpha}(-t)v)\\
& \ \ =\sum_{s+t=-n}(u,-\mbox{$\circ\atop\circ$}h_\alpha(s)h_{\alpha}(t)\mbox{$\circ\atop\circ$}v+k_{\alpha}\mbox{$\circ\atop\circ$}x_{\alpha}(s)x_{-\alpha}(t)\mbox{$\circ\atop\circ$}v
+k_{\alpha}\mbox{$\circ\atop\circ$}x_{-\alpha}(s)x_{\alpha}(t)\mbox{$\circ\atop\circ$}v)\\
& \ \ = 2k_{\alpha}(k_{\alpha}+2)(u,L_{\alpha}(-n)v),
\end{split}
\]
as desired.
\end{proof}

We are now proving Theorem \ref{t2}.
\begin{proof} For short, we set $V=K(\g,k).$ It
follows from \cite{L2} that if ${\mathcal{D}}(V)=V$, then $V$ is
$C_2$-cofinite. Assume ${\mathcal{D}}(V)\not =V.$ Then
there exists $v=(v^0,...,v^n,...)\in
{\mathcal{D}}(V)$ not in $V$, where $v^i\in V_{i}=K(\g,k)_{i}.$
That is, there are infinitely many nonzero $v^i.$ Recall that
$k_{\alpha}=\frac{2}{\<\alpha,\alpha\>}k.$ So
$k=\frac{\<\alpha,\alpha\>}{2}k_{\alpha}$. Assume that
$k_{\alpha}\leq q$ for all $\alpha\in \Delta_+,$ then $P_{\a}$ is
regular for all $\alpha\in \Delta_+$ from the assumption. Thus the
weak $P_{\alpha}$-module ${\mathcal{D}}(V)$ is a a direct sum of
irreducible ordinary $P_{\alpha}$-modules.

It is easy to see that if  $\g$ is $ADE$ type, then $k \leq q,$ if
$\g$ is type $G_2$, then $k\leq [q/3],$ and if $\g$ is the other
type, $k\leq [q/2].$ Since
 ${\mathcal{D}}(V)$ is a a direct sum of irreducible ordinary
$P_{\alpha}$-modules, each $L_{\alpha}(0)$ is semisimple on
${\mathcal{D}}(V).$  This implies that $v$ is a sum of
eigenvectors for $L_{\a}(0)$ with rational eigenvalues
$\lambda^{\a}_1,...,\lambda^{\alpha}_{j_{\alpha}}$ \cite{DLM3}.
 Since $L_{\a}(0)$ preserves each $V_i$
for all $i$, we see that each $v^i$ is a sum of eigenvectors for
$L_{\a}(0)$ with possible eigenvalues
$\lambda^{\a}_1,...,\lambda^{\alpha}_{j_{\alpha}}.$
 Let $\lambda_{\alpha}$ be the maximum of $\lambda^{\alpha}_j$ for
$1\leq j\leq j_{\alpha}.$

Recall that $L(0)=\sum_{\alpha\in\Delta_+}c_{\alpha}L_{\a}(0)$,
where $c_{\a}=\frac{k(k_{\alpha}+2)}{k_{\alpha}(k+h^{\vee})}$ is
positive. It is clear that $(L(0)v^i,v^i)=i(v^i,v^i).$ By Lemma
\ref{l1}, each $L_{\alpha}(0)$ is a hermitian operator on $V_i$
and eigenvectors with different eigenvalues are orthogonal with
respect to the positive definite hermitian form $(,).$ As a
result, $(L_{\alpha}(0)v^i,v^i)\leq \lambda_{\alpha}(v^i,v^i)$ for
all $i.$  This yields
\[
i(v^i,v^i)=(L(0)v^i, v^i)=\sum_{\alpha\in\Delta_+}c_{\alpha}(L_{\a}(0)v^i,v^i)
\leq \sum_{\alpha\in\Delta_+}c_{\alpha}\lambda_{\alpha}(v^i,v^i).
\]
for all $i.$ Since there are infinitely many nonzero $v^i$, we
conclude that there are infinitely many positive integers $i$ less
than or equal to a fixed number
$\sum_{\alpha\in\Delta_+}c_{\alpha}\lambda_{\alpha}.$ This is
obviously a contradiction. The proof is complete.
\end{proof}

\section{The structure of $K(\g,1)$}
\def\theequation{4.\arabic{equation}}
\setcounter{equation}{0}

In this section, we will discuss the parafermion vertex operator
algebras for $k=1.$ We need the following Lemma  from \cite{DLY2}.
\begin{lem}\label{l4.1} The parafermion vertex operator algebras $K(sl_2,1)=\C$
and  $K(sl_2,2)$ is isomorphic to $L(\frac{1}{2},0).$ In
particular,  $W^3_\a=0$ in both cases, where $\a$ is a root of
$sl_2$.
\end{lem}

From Lemma \ref{l4.1}, if $k=1$ and  $\alpha$ is a long root, then
$\omega_{\alpha}=W^3_{\alpha}=0.$ This shows that $K(\g,1)=\C$ if
$\g$ is of $ADE$ types. We will restrict ourselves to the
non-simply laced simple Lie algebra in the rest of this section.
Recall that $\<\theta,\theta\>=2.$ We note that the subalgebra
$P_{\alpha}$ of $K(\g,1)$ generated by $\omega_{\alpha},
W^3_{\alpha}$ with $\a\in \Delta_+$ is isomorphic to $K(sl_2,2)$
if $\<\alpha,\alpha\>=1$ and isomorphic to $K(sl_2,3)$ if
$\<\alpha,\alpha\>=\frac{2}{3}.$ Furthermore, $K(sl_2,2)$ is
isomorphic to the rational vertex operator algebra
$L(\frac{1}{2},0)$ which is the irreducible highest weight module
for the Virasoro algebra with central charge $\frac{1}{2}.$ So in
the case that $\g$ is of type $B_l,C_l,F_4,$ $K(\g,1)$ is
generated by $\omega_{\alpha}$ with $\alpha\in \Delta_+$ being
short roots. If $\g$ is of type $G_2,$ the situation is more
complicated. The result is as follows: the parafermion vertex
operator algebra associated to the non-simply laced simple Lie
algebra with $k=1$ is the same as the parafermion vertex operator
algebra obtained from the simply laced simple Lie algebra whose
Dynkin diagram is obtained from the Dynkin diagram of the
non-simply laced algebra by deleting the long roots with $k=3$ in
the case of $G_2$ and with $k=2$ in the rest of cases. This result
has been observed in \cite{G} by using the partition functions.
Recall that the central charge of the parafermion vertex operator
algebra is given by $\frac{klh-lh^{\vee}}{k+h^{\vee}}.$

In the following, we fix an orthonormal basis $\{\e_1,...,\e_l\}$
of Euclidean space $\R^n$ and we refer the reader to \cite{H} for
the details on root systems. We denote the set of short roots by
$\Delta^s$ and the set of positive short roots by $\Delta_+^s.$

Recall that $Q$ is the root lattice of $\g.$ Denote by $Q_L$ the sublattice generated by
the long roots. The following lemma is important in our discussion
below.
\begin{lem}\label{l4.2} Assume that $\alpha,\beta\in \Delta^s$ such that $\alpha\in Q_L\pm
\beta$, then $\omega_{\alpha}=\omega_{\beta}$, $W^3_{\alpha}=\pm
W^3_{\beta}$ in $K(\g,1).$
\end{lem}
\begin{proof} The result is obvious if $\alpha=\pm \beta.$
We only need to deal with the case that $\a\ne \pm\beta.$ Suppose that
$\alpha\in Q_L+\beta.$

We first prove that
 $\gamma=\alpha-\beta$ is a long root. Note that $\<\gamma,\gamma\>=\<\alpha,\alpha\>+\<\beta,\beta\>-2\<\alpha,\beta\>$
  is a positive even integer as $Q_L$ is a positive definite even lattice.
If $\g$ is of type of $G_2,$ then
$\<\gamma,\gamma\>=\frac{4}{3}-2\<\alpha,\beta\>$ is an even
integer greater than or equal to $2.$ This forces
$\<\alpha,\beta\>$ to be $-\frac{1}{3}$ and $\gamma$ is a long
root. If $\g$ is not of type $G_2,$
$\<\gamma,\gamma\>=2-2\<\alpha,\beta\>.$ The possible values for
 $ \<\alpha,\beta\>$ is $0$ or $-1.$ If  $\<\alpha,\beta\>=-1$, then $\alpha=-\beta$
 and this is impossible from the assumption. So $\<\alpha,\beta\>=0$ and
$\<\gamma,\gamma\>=2.$ Again $\gamma$ is a root.

Let $\sigma_{\gamma}$ be the reflection associated to root
$\gamma.$ It is easy to verify that
$\sigma_{\gamma}(\alpha)=\beta.$ Let
$\tau_{\gamma}=e^{x_{-\gamma}(0)}e^{-x_{\gamma}(0)}e^{x_{-\gamma}(0)}.$
Then $\tau_{\gamma}$ is an automorphism of ${\mathcal{L}}(1,0).$
Since $\tau_{\gamma}$ preserves $h(-1)\1$ for $h\in \h$, we see
that $\tau_{\gamma}$ preserves $K(\g,1).$ In fact,
$\tau_{\gamma}(\omega_{\alpha})=\omega_{\beta}$,
$\tau_{\gamma}(W^3_{\alpha})=W^3_{\beta}$. So it is sufficient to
show that $\tau_{\gamma}=1$ on $K(\g,1).$

Let $\Delta^l$ be a subset of $\Delta$ consisting of long roots.
Then $\Delta^l$ is a root lattice which is orthogonal union of
irreducible root systems of ADE types. Consider the vertex operator
subalgebra $U$ of ${\mathcal{L}}(1,0)$ generated by
$x_{\phi}(-1)\1$ for $\phi\in \Delta^l.$ Then $U=V_{Q_L}$ is a
lattice vertex operator algebra. It is well known that the
Virasoro element of $U$ is given by
$$\omega_U=\frac{1}{2}\sum_{i=1}^su^i(-1)^2\1$$
where $\{u^1,...,u^s\}$ is an orthonormal basis of $\sum_{\phi\in
\Delta^l}\C h_{\phi}$ with respect to $\<.\>.$ Since
$$u^i(n)K(\g,1)=0$$
 for $i=1,...,s$ and $n\geq 0$, we see that
$L_U(-1)K(\g,1)=0$, where $L_U(-1)$ is the component operator of
$Y(\omega_U,z)=\sum_{m\in \Z}L_U(m)z^{-m-2}.$ This implies that
$u_mK(\g,1)=0$ for all $u\in U$ and $m\geq 0.$ In particular,
$x_{\gamma}(0)=x_{-\gamma}(0)=0$ on $K(\g,1).$ As a result,
$\tau_{\gamma}=1$ on  $K(\g,1).$ So we have proved that if
$\a+Q_L=\b+Q_L$, then $\omega_{\a}=\omega_{\b},
W^3_{\alpha}=W^3_{\beta}$.

Similarly, if $\a\in Q_L-\b$, then
$\omega_{\a}=\omega_{-\b}=\omega_{\b}$,
$W^3_{\alpha}=W^3_{-\beta}=-W^3_{\beta}$. The proof is complete.
\end{proof}

We also need to generalize (3) of Theorem \ref{t1} to an arbitrary
subalgebra of $\g.$
\begin{lem}\label{l4.3} Let $\g_1$ be simple subalgebra of $\g$ with root
system $\Delta_1\subset \Delta.$ Then the vertex operator
subalgebra of $K(\g,k)$ generated by $\omega_{\a}, W^3_{\a}$ for
$\alpha\in \Delta_1$ is isomorphic to $K(\g_1,k_1),$  where $k_1$
is defined as follows: $k_1=k$ if $\Delta_1\nsubseteq \Delta^s;$
$k_1=3k$ if $\g$ is of type $G_2$ and $\Delta_1\subset \Delta^s$;
$k_1=2k$ if $\g$ is of type $B_l,C_l,F_4$ and $\Delta_1\subset
\Delta^s$.
\end{lem}

\begin{proof} The proof is similar to that of Proposition 4.6 of \cite{DW}
by noting that for each $\alpha\in \Delta$, ${\cal L}(k,0)$ is an
integrable module for the affine Lie algebra $\widehat{\g_\a}.$
\end{proof}

We are now ready to discuss the parafermion vertex operator
algebra $K(\g,1)$ if $\g$ is a non-simply laced simple Lie algebra.
We denote the root system whose simple roots are the short simple
roots of the non-simply laced Lie algebra by $\Delta'.$

\subsection{$B_l$} The root system is given by
$$\Delta=\{\pm \e_i, \pm(\e_i\pm \e_j)|i,j=1,...,l, i\ne j\}$$
with simple roots
$$\{\e_1-\e_2,...,\e_{l-1}-\e_l, \e_l\}$$
and
$\Delta'=\{\pm \e_l\}$  is the root system of type $A_1.$

 In this case, $K(B_l,1)$ is generated by $\omega_{\e_i}$ for $i=1,...,l$
and each $\omega_{\e_i}$ generates a vertex operator algebra
isomorphic to $L(\frac{1}{2},0).$ Note that the central charge of
$K(B_l,1)$ is equal to $\frac{1}{2}.$ So $K(B_l,1)$ is an
extension of $L(\frac{1}{2},0).$ But the only extension of
$L(\frac{1}{2},0)$ is itself due to the integral weight
restriction. As a result, $K(B_l,1)=L(\frac{1}{2},0)$ and
$\omega_{\e_i}=\omega$ for all $i.$ One can also use Lemma
\ref{l4.2} to see that $\omega_{\e_i}=\omega_{\e_j}$ for all
$i,j.$

\subsection{$C_l$}\label{s4.2} Assume that $l\geq 3.$ The root system is given by
$$\Delta=\{\pm 2\e_i, \pm(\e_i\pm \e_j)|i,j=1,...,l, i\ne j\}$$
with simple roots
$$\{\e_1-\e_2,...,\e_{l-1}-\e_l, 2\e_l\}$$
and
$$\Delta'=\{\e_i-\e_j|i\ne j\}$$
is a root system of type $A_{l-1}.$ By Lemma \ref{l4.2}, we see
that $\omega_{\e_i-\e_j}=\omega_{\e_i+\e_j}$ if $i\ne j.$ Thus
$K(C_l,1)$ is generated by $\omega_{\e_i-\e_j}$ for $i<j.$

Note that
$$\g_1=\sum_{i\ne j}(\C x_{\e_i-\e_j}+\C h_{\e_i-\e_j})$$
is a subalgebra of $C_l$ isomorphic to $A_{l-1}.$ As a result, the
vertex operator algebra $K(C_l,1)$ is isomorphic to $K(A_{l-1},2)$
by Lemma \ref{l4.3}.

\subsection{$F_4$} The root system is given by
$$\Delta=\{\pm \e_i, \pm(\e_i\pm \e_j), \pm\frac{1}{2}(\e_1\pm\e_2\pm \e_3\pm \e_4)|i,j=1,...,4, i\ne j\}$$
with simple roots
$$\{\e_2-\e_3, \e_3-\e_4, \e_4, \frac{1}{2}(\e_1-\e_2-\e_3-\e_4)\}$$
and
$$\Delta'=\{\pm \e_4, \pm\frac{1}{2}(\e_1-\e_2-\e_3\pm \e_4)\}$$
is a root system of type $A_2.$ By Lemma \ref{l4.3} again,
$K(F_4,1)$ is isomorphic to $K(A_2,2).$

We now determine $K(A_2,2)$ explicitly. Let
$\{\alpha_1,\alpha_2\}$ be the simple roots of $\Delta.$ Set
$\omega'=\omega-\omega_{\a_1}.$ Then $\omega'$ is a Virasoro
vector with central charge $\frac{7}{10}.$ By Lemma \ref{l1},
$K(A_2,2)$ is a unitary representations for both $Vir_{\a_1}$ and
the Virasoro algebra generated by the component operators of
$Y(\omega',z).$ Let $U$ be a vertex operator subalgebra of
$K(A_2,2)$ generated by $\omega_{\alpha_1}$ and $\omega'.$ Then
$U$ is isomorphic to rational vertex operator algebra
$L(\frac{1}{2},0)\otimes L(\frac{7}{10},0).$ So $K(A_2,2)$ is a
completely reducible $U$-module. Note that the irreducible
$L(\frac{7}{10},0)$-modules are $L(\frac{7}{10},h)$ with
$$h=0,\frac{7}{16}, \frac{3}{80}, \frac{3}{2}, \frac{3}{5}, \frac{1}{10}$$
\cite{W} and $L(\frac{1}{2},0)$ is a rational vertex operator
algebra with 3 irreducible modules $L(\frac{1}{2},h')$ with
$h'=0,\frac{1}{2},\frac{1}{16}$ \cite{DMZ}, \cite{W}. Since
$K(A_2,2)$ only has integral weight, we see that as a $U$-module,
$K(A_2,2)$ can only be $U$ or $U\oplus
n(L(\frac{1}{2},\frac{1}{2})\otimes L(\frac{7}{10},\frac{3}{2})).$
But $K(A_2,2)_2$ is $3$-dimensional with basis $\omega_{\a_1},
\omega_{\a_2}, \omega_{\a_1+\a_2}$, thus we conclude that
$$K(A_2,2)=L(\frac{1}{2},0)\otimes L(\frac{7}{10},0)\oplus L(\frac{1}{2},\frac{1}{2})\otimes  L(\frac{7}{10},\frac{3}{2}).$$

\subsection{$G_2$} The root system is given by
$$\Delta=\{\pm(\e_1-\e_2),\pm(\e_2-\e_3), \pm(\e_1-\e_3), \pm (2\e_1-\e_2-\e_3),\pm(2\e_2-\e_1-\e_3), \pm(2\e_3-\e_1-\e_2)\}$$
with simple roots
$$\{\e_1-\e_2, -2\e_1+\e_2+\e_3\}$$ and $\Delta'=\{\pm (\e_1-\e_2)\}$
is a root system of type $A_1.$ By Lemma \ref{l4.2}, we deduce
that $K(G_2,1)$ is generated by $\omega_{\e_1-\e_2}$ and
$W_{\e_1-\e_2}^3$. Thus by Lemma \ref{l4.3}, $K(G_2,1)$ is
isomorphic to $K(A_1,3)$ which is isomorphic to
$L(\frac{4}{5},0)\oplus L(\frac{4}{5},3)$ \cite{DLY2}. The vertex
operator algebra $K(A_1,3)$ is rational and the irreducible
modules has been classified in \cite{KMY}.

\end{document}